\documentclass[12pt,verbatim]{amsart}
\usepackage{amsfonts}
\usepackage{ifthen}
\usepackage{amsthm}
\usepackage{amsmath}
\usepackage{amssymb}
\usepackage{graphicx}
\usepackage{subfigure}
\usepackage{amscd,amssymb,amsthm}
\usepackage{color,xcolor}
\newcounter{minutes}
\divide\time by 60
\newcounter{hours}
\multiply\time by 60 \addtocounter{minutes}{-\time}
\newtheorem{theorem}[equation]{Theorem}

\newtheorem{lemma}[equation]{Lemma}

\newtheorem{remark}[equation]{Remark}

\newtheorem{nonsec}[equation]{}
\newtheorem{conjecture}[equation]{Conjecture}

\newcommand{\beq}{\begin{equation}}
\newcommand{\eeq}{\end{equation}}

\newcommand{\R}{\mathbb{R}^2}

\newcommand{\UH}{\mathbb{H}^2}
\newcommand{\BB}{\mathbb{B}^2}

\newcommand{\Hn}{ {\mathbb{H}^n} }
\newcommand{\Bn}{ {\mathbb{B}^n} }
\newcommand{\Rn}{ {\mathbb{R}^n} }

\newcommand{\arth}{\,\textnormal{arth}}

\numberwithin{equation}{section}

\begin{document}

%%%%%%%%%%%%%%%%%%%%%%%%%%%%%%%%%%%%
%\newpage
%\thispagestyle{empty}

%\noindent {\Large\bf Publication IV }

%\vspace{.5cm}

%\noindent {\sc  S. Simi\'c, M. Vuorinen, and G.-D. Wang}: {\it  Sharp Lipschitz constants for the distance ratio metric.}
%Math. Scand.(to appear), arXiv:1202.6565v4 [math. CV].

%%%%%%%%%%%%%%%%%%%%%%%%%%%%%%%%%%%%
%\newpage
%\thispagestyle{empty}
%\mbox{ }

%%%%%%%%%%%%%%%%%%%%%%%%%%%%%%%%%%%%
%%%%%%%%%%%%%%%%%%%%%%%%%%%%%%%%%%%%

%\newpage
%\thispagestyle{headings}

%===============================================================================
%%%%%%%% BEGIN TIMESTAMP
\def\thefootnote{}
\footnotetext{ \texttt{\tiny File:~\jobname .tex,
          printed: \number\year-\number\month-\number\day,
          \thehours.\ifnum\theminutes<10{0}\fi\theminutes}
} \makeatletter\def\thefootnote{\@arabic\c@footnote}\makeatother
%%%%%%%% END TIMESTAMP
%===============================================================================

%=====================================================================
\title{Sharp Lipschitz constants for the distance ratio metric}

\author{Slavko Simi\'c}
\author{Matti Vuorinen}
\author{Gendi Wang}

\address{Mathematical Institute SANU, Kneza Mihaila 36, 11000 Belgrade, Serbia}
\email{ssimic@turing.mi.sanu.ac.rs}

\address{Department of Mathematics and Statistics, University of Turku, Turku 20014,
Finland}\email{vuorinen@utu.fi}

\address{Department of Mathematics and Statistics, University of Turku, Turku 20014,
Finland}
\email{genwan@utu.fi}

%=========================================================================

\maketitle

%===============================================================================
\begin{abstract}
We study expansion/contraction properties of some common classes of
mappings of the Euclidean space ${\mathbb R}^n, n\ge 2\,,$
with respect to the distance ratio metric. The first main case is  the behavior of M\"obius transformations of the unit ball in ${\mathbb R}^n$ onto itself.
In the second main case we study the %behavior of M\"obius transformations of the upper half plane onto the unit disk and
polynomials of the unit disk onto a subdomain of the complex plane. In both cases sharp Lipschitz constants are
obtained.
%We find sharp Lipschitz constants for the distance ratio metric under %M\"obius transformations of the unit ball in ${\mathbb R}^n$ and for %$n=2$ under analytic functions or polynomials of the unit disk.
\end{abstract}
%===============================================================================

%\pagestyle{plain}

%===============================================================================
\section{Introduction}
%===============================================================================

Conformal invariants and conformally invariant metrics have been some of the key notions of geometric function theory and of quasiconformal
mapping theory for several decades.
One of the modern trends is to extend this theory to Euclidean spaces
of higher dimension or to more general metric spaces and new methods have been invented. It has turned out that we cannot any more expect
the same invariance properties as in the classical cases, but still
some type of "nearinvariance" or "quasi-invariance" is a desirable
feature. The quasihyperbolic metric of a domain $G \subset {\mathbb R}^n, n\ge 2\,$ is one of these new notions and it is known to be quasi-invariant in the sense described below and there are many reasons why one may regard it as a spatial version of the hyperbolic
metric. This notion is also applicable in the wider context of metric spaces. For example, J. V\"ais\"al\"a's theory of quasiconformal maps in infinite dimensional Banach spaces is
entirely based on the quasihyperbolic metric \cite{va1}.
 Another crucial notion, which is the object of the present
study, is the distance ratio metric. It is often applied to study
quasihyperbolic geometry. Very little is known about the geometric
properties of balls in these two spaces: only recently it was proved
for instance that in some special cases
 for small radii the balls are convex \cite{k1,k2,mv,va2,rt}
no matter where the center is. This convexity property is expected
to hold in general, but it has not been proved yet.

We next define these two metrics.

{\bf Distance ratio metric.}
For a proper subset $G$ of $\Rn$ and for all
$x,y\in G$, the  distance-ratio
metric $j_G$ or $j$-metric is defined as
\begin{eqnarray*}
 j_G(x,y)=\log \left( 1+\frac{|x-y|}{\min \{d(x,\partial G),d(y, \partial G) \} } \right)\,,
\end{eqnarray*}
where $d(x,\partial G)$ denotes the Euclidean distance
from the point $x$ to the boundary  $\partial G$ of the domain $G\,.$
The distance ratio metric was introduced by F.W. Gehring and B.P. Palka
\cite{gp} and in the above simplified form by M. Vuorinen \cite{v0}. Both definitions are
frequently used in the study of hyperbolic type metrics \cite{himps},
geometric theory of functions \cite{v}, and quasiconformality in Banach spaces \cite{va1}.

Let $p\in G$, then for all $x,y\in G\setminus\{p\}$
\begin{eqnarray*}
 j_{G\setminus\{p\}}(x,y)=\log \left( 1+\frac{|x-y|}{\min \{d(x,\partial G),d(y, \partial G), |x-p|, |y-p|\} } \right)\,.
\end{eqnarray*}
This formula shows that the $j$-metric highly depends on the boundary of the domain.

\medskip

{\bf Quasihyperbolic metric.}
Let $G$ be a proper subdomain of ${\mathbb R}^n\,$. For all $x,\,y\in G$, the quasihyperbolic metric $k_G$ is defined as
$$k_G(x,y)=\inf_{\gamma}\int_{\gamma}\frac{1}{d(z,\partial G)}|dz|,$$
where the infimum is taken over all rectifiable arcs $\gamma$ joining $x$ to $y$ in $G$.

\medskip
%It should be noted that these metrics can be estimated  one in terms
%of the other and the hyperbolic metrics, $\rho_\Bn$ and $\rho_\Hn$, of the unit ball $\Bn$ and the half space $\Hn$, resp.
It should be noted that in the case $G=\Bn$ or $G=\Hn$ the three metrics $\rho_{G}$, $k_G$, and $j_G$ can be compared.
For instance, below we shall apply the following basic inequalities \cite[Lemma 2.41]{v} and \cite[Lemma 7.56]{avv}:
%$$ j_G(x,y)\le k_G(x,y)\,, \forall  \,\, x,y \in G\,,$$
%\begin{equation} \label{jrho}
\begin{eqnarray}
 \label{jrho1}
 \frac{1}{2} \rho_{\mathbb{B}^n}(x,y) &\le & j_{\mathbb{B}^n}(x,y) \le   k_{\mathbb{B}^n}(x,y) \le \rho_{\mathbb{B}^n} (x,y)   \,, \forall  \,\, x,y \in \mathbb{B}^n \,,\\
 \label{jrho2}
\frac{1}{2}\rho_{\mathbb{H}^n}(x,y) &\le & j_{\mathbb{H}^n}(x,y) \le   k_{\mathbb{H}^n}(x,y) \equiv \rho_{\mathbb{H}^n} (x,y)   \,, \forall  \,\, x,y \in \mathbb{H}^n \, .
\end{eqnarray}
%\end{equation}
\medskip

The hyperbolic metric in the unit ball or half space is M\"obius invariant. However, neither
the quasihyperbolic metric nor the distance ratio metric
is invariant under M\"obius transformations. Therefore, it is natural to ask what the Lipschitz constants are for these metrics under conformal mappings or M\"obius transformations in higher dimension.
F.\,W. Gehring, B.\,G. Osgood, and B.\,P. Palka proved that these metrics are not
changed by more than a factor $2$ under M\"obius transformations, see \cite[proof of Theorem 4]{go} and \cite[Corollary 2.5]{gp}:

%==============================================================================
\begin{theorem} \label{gplem}
If $D$ and $D'$ are proper subdomains of $\Rn$ and if $f$ is a
M\"obius transformation of $D$ onto $D'$, then for all $x,y\in D$
$$\frac12 m_{D}(x,y)\leq m_{D'}(f(x),f(y))\leq 2m_D(x,y),$$
where $m\in\{j,k\}$.
\end{theorem}
%==============================================================================

R. Kl\'en, M. Vuorinen, and X.-H. Zhang studied the sharpness of the constant $2$ in Theorem \ref{gplem}. They got the sharp bilipschitz constant $1+|a|$ for the quasihyperbolic metric under M\"obius self-mappings of the unit ball \cite[Theorem 1.4]{kvz}, and proposed a conjecture for the distance ratio metric.
%==============================================================================
\begin{conjecture} \cite[Conjecture 2.3]{kvz}\label{con}
Let $a\in\Bn$ and $h : \Bn \to \Bn=h\Bn$ be a M\"obius transformation with $h(a) =0\,.$
Then
$$\sup_{x,y\in\Bn\atop{x\neq y}}\frac{j_{\Bn}\left(h(x),h(y)\right)}{j_{\Bn}(x,y)}=1+|a|.$$
\end{conjecture}
%==============================================================================

The following Theorem \ref{mth1} shows that the  solution to Conjecture \ref{con} is in the affirmative and thus the Gehring--Osgood--Palka Theorem \ref{gplem} admits a refinement for M\"obius transformations of the unit ball onto itself.
%===============================================================================
\begin{theorem} \label{mth1}
Let $a\in {\mathbb B}^n$ and $f: {\mathbb B}^n \to {\mathbb B}^n= f{\mathbb B}^n$  be a M\"obius transformation with
$f(a)=0$. Then for all $x,y \in {\mathbb B}^n$
\begin{eqnarray*}
\frac{1}{1+|a|}j_{{\mathbb B}^n}(x,y)\leq j_{{\mathbb B}^n}(f(x),f(y)) \leq (1+ |a|) j_{{\mathbb B}^n}(x,y),
\end{eqnarray*}
and the constants $\frac{1}{1+|a|}$ and $1+|a|$ are both the best possible.
\end{theorem}
%===============================================================================
 We also study how the $j$-metric behaves under
 polynomials, or more generally,
 analytic functions of the unit disk onto a subdomain of the complex plane. In this direction, our main results are the following two theorems.
 %==============================================================================
\begin{theorem}\label{sth5}
Let $p\in \mathbb N$ and $\{a_l\}$ be a sequence of complex numbers with $\sum_{l=1}^{p}|a_l|\le
1$. Let $f: \BB\rightarrow\BB$ with $f(\BB\setminus \{0\})\subset \BB\setminus \{0\}$ and $f(z)=\sum_{l=1}^{p}a_l z^l$ and $f(0)=0.$ Then for all $x,\,y\in \BB\setminus \{0\}$
\begin{eqnarray*}
j_{\BB\setminus \{0\}}(f(x),f(y))\le p j_{\BB\setminus \{0\}}(x,y),
\end{eqnarray*}
and the constant $p$ is sharp.
\end{theorem}
%==============================================================================

%===============================================================================
\begin{theorem} \label{mth3}
Let $\{a_l\}$ be a sequence of complex numbers and $\sum_{l=0}^{\infty}|a_l|\le 1$. Let $f: {\mathbb B}^2 \to {\mathbb B}^2$ be a non-constant analytic function defined by $f(z)=\sum_{l=0}^{\infty}a_l z^l $. Then for all
 $x,\,y\in {\mathbb B}^2$
\begin{eqnarray*}
j_{{\mathbb B}^2}(f(x),f(y))\le j_{{\mathbb B}^2}(x,y),
\end{eqnarray*}
and this inequality is sharp.
\end{theorem}
%===============================================================================

%%%%%%%%%%%%%%%%
\section{Lipschitz Constants under some special mappings}
%%%%%%%%%%%%%%%%
%%%%%%%%%%%%%%%%
%%%%%%%%%%%%%%%%
%%%%%%%%%%%%%%%%
We recall here some basic facts about M\"obius transformations from \cite{b, v}.

\medskip
{\bf M\"obius transformations.}
The group of M\"obius transformations in $\overline{\Rn}$ is generated by transformations of two types:

(1) reflections in the hyperplane $P(a,t)=\{x\in \Rn: x\cdot a=t\}\cup\{\infty\}$
$$f_1(x)=x-2(x\cdot a-t)\frac{a}{|a|^2},\,\, f_1(\infty)=\infty,$$
where $a\in\Rn\setminus\{0\}$ and $t\in\R$;

(2) inversions (reflections) in the sphere $S^{n-1}(a,r)=\{x\in \Rn: |x-a|=r\}$
$$f_2(x)=a+\frac{r^2(x-a)}{|x-a|^2},\,\,f_2(a)=\infty, f_2(\infty)=a,$$
where $a\in\Rn$ and $r>0$. If $G\subset\overline{\Rn}$ we denote by $\mathcal{GM}(G)$ the group of all M\"obius transformations which map $G$ onto itself.

We denote $a^*=\frac{a}{|a|^2}$ for $a\in\Rn\setminus\{0\}$, and $0^*=\infty$, $\infty^*=0$. For fixed $a\in\Bn\setminus\{0\}$, let
\begin{eqnarray}\label{iv}
\sigma_a(z)=a^*+r^2(z-a^*)^*,\,\,r^2=|a|^{-2}-1
\end{eqnarray}
be an inversion in the sphere $S^{n-1}(a^*,r)$ orthogonal to $S^{n-1}$. Then $\sigma_a(a)=0$,  $\sigma_a(a^*)=\infty$ and by \cite[3.1.5]{b},\cite[(1.5)]{v}
\begin{eqnarray}\label{ivd}
|\sigma_a(x)-\sigma_a(y)|=\frac{r^2|x-y|}{|x-a^*||y-a^*|}.
\end{eqnarray}

{\bf Lipschitz mappings.}
 Let $(X,d_X)$ and $(Y,d_Y)$ be metric spaces. Let $f: X\rightarrow Y$ be continuous and let $L\geq1$. We say that $f$ is $L$-lipschitz if
\begin{eqnarray*}
d_Y(f(x),f(y))\leq L d_X(x,y),\,\,{\rm for}\,\, x,\,y\in X,
\end{eqnarray*}
and $L$-bilipschitz if $f$ is
a homeomorphism and
\begin{eqnarray*}
d_X(x,y)/L\leq d_Y(f(x),f(y))\leq L d_X(x,y),\,\,{\rm for}\, \,x,\,y\in X.
\end{eqnarray*}
A $1$-bilipschitz mapping is called an isometry.

\medskip

The next lemma, so-called {\em
monotone form of l'H${\rm \hat{o}}$pital's rule}, has found recently numerous applications in proving inequalities. See the extensive bibliography of \cite{avz}.

%==============================================================================
\begin{lemma} \label{lhr}{\rm \cite[Theorem 1.25]{avv}}
For $-\infty<a<b<\infty$, let $f,\,g: [a,b]\rightarrow \mathbb{R}$ be continuous on $[a,b]$, and be differentiable on $(a,b)$, and let $g'(x)\neq 0$ on $(a,b)$. If $f'(x)/g'(x)$ is increasing (deceasing) on $(a,b)$, then so are
\begin{eqnarray*}
\frac{f(x)-f(a)}{g(x)-g(a)}\,\,\,\,\,\,\,and\,\,\,\,\,\,\,\,\frac{f(x)-f(b)}{g(x)-g(b)}.
\end{eqnarray*}
If $f'(x)/g'(x)$ is strictly monotone, then the monotonicity in the conclusion is also strict.
\end{lemma}
%==============================================================================

%==============================================================================
\begin{theorem}\label{sth1}
Let $f:\UH \rightarrow \BB$ with $f(z)=\frac{z-i}{z+i}$.

(1) For all $x,y\in \UH$
$$j_{\BB}(f(x),f(y))\leq 2 j_{\UH}(x,y),$$
and the constant 2 is the best possible.

(2) For all $x,y\in\UH\setminus\{i\}$
$$j_{\BB\setminus\{0\}}(f(x),f(y))\leq 2 j_{\UH\setminus\{i\}}(x,y), $$
and the constant 2 is the best possible.
\end{theorem}
%==============================================================================
%==============================================================================
\begin{proof}
It suffices to show the sharpness of the inequalities by Theorem \ref{gplem}.

(1) Let $x=t+ai$ and $y=ai$, where $a>0$ and $t>0$. Then
\begin{eqnarray*}
\lim_{t\rightarrow\infty}\frac{j_{\BB}(f(t+ai),f(ai))}{j_{\UH}(t+ai,ai)}&=&\lim_{t\rightarrow\infty}\frac{\log\left(1+\frac{t}{2a(a+1)}\left(\sqrt{t^2+(a+1)^2}+\sqrt{t^2+(a-1)^2}\right)\right)}
{\log(1+\frac ta)}\\
&=&\lim_{t\rightarrow\infty}\frac{\log\left(\frac{2t^2}{2a(a+1)}\right)}{\log(\frac ta)}\\
&=&2.
\end{eqnarray*}
Hence the constant 2 is the best possible.

\medskip

(2) Let $x=t+ai$ and $y=ai$, where $0<a<\frac 13$ and $t>1$. Then
\begin{eqnarray*}
\frac{j_{\BB\setminus\{0\}}(f(x),f(y))}{j_{\UH\setminus\{i\}}(x,y)}=\frac{j_{\BB}(f(x),f(y))}{j_{\UH}(x,y)}.
\end{eqnarray*}
By (1) we obtain that the constant 2 is the best possible.
\end{proof}
%==============================================================================

\medskip

%==============================================================================
\begin{theorem} \label{sth2}
Let $f: \BB\rightarrow \UH$ with $f(z)=i\frac{1+z}{1-z}$.

(1) For all $x,\, y\in \BB$
$$
j_{\UH}(f(x),f(y))\le 2 j_{\BB}(x,y)\,,
$$
and the constant $2$ is the best possible.

(2) For all $x,\, y\in \BB\setminus \{0\}$
$$j_{\UH\setminus \{i\}}(f(x),f(y))\le 2 j_{\BB\setminus \{0\}}(x,y)\,,$$
and the constant 2 is the best possible.
\end{theorem}
%==============================================================================
%==============================================================================
\begin{proof}
By Theorem \ref{gplem}, we only need to show that $2$ is the best possible.

(1) For $t=x=-y, \ t\in (0,1)$, we get
\begin{eqnarray*}
j_{{\UH}}(f(x),f(y))&=&\log\left(1+\frac{4t}{(1-t)^2}\right)\\
&=&2\log\left(1+\frac{2t}{1-t}\right)\\
&=&2j_{\BB}(x,y).
\end{eqnarray*}
Therefore, the constant $2$ is the best possible.

\medskip

(2) Since for $t=x=-y$ and $\frac{1}{2}<t<1$, we have
$$
j_{\UH\setminus \{i\}}(f(x),f(y))=j_{\UH}(f(x),f(y))=2 j_{\BB}(x,y)=2 j_{\BB\setminus \{0\}}(x,y),
$$
the constant $2$ is the best possible.
\end{proof}
%==============================================================================

\medskip

%===============================================================================
\begin{theorem} \label{mth2}
Let $a\in {\mathbb H}^2$ and $f: {\mathbb H}^2 \to \BB=f\UH$  be a M\"obius transformation with
$f(a)=0$.

(1) For all $x,y \in {\mathbb H}^2$
\begin{eqnarray*}
\frac{1}{2}j_{{\mathbb H}^2}(x,y)\leq j_{{\mathbb B}^2}(f(x),f(y)) \leq 2 j_{{\mathbb H}^2}(x,y),
\end{eqnarray*}
and the constants $\frac 12$ and $2$ are both the best possible.

(2) For all $x,y \in {\mathbb H}^2\setminus\{a\}$
\begin{eqnarray*}
\frac{1}{2}j_{{\mathbb H}^2\setminus\{a\}}(x,y)\leq j_{{\mathbb B}^2\setminus\{0\}}(f(x),f(y)) \leq 2 j_{{\mathbb H}^2\setminus\{a\}}(x,y),
\end{eqnarray*}
and the constants $\frac 12$ and $2$ are both the best possible.
\end{theorem}
%===============================================================================

%==============================================================================
\begin{proof}
 A M\"obius transformation satisfying the assumptions is of the form
$$
f(z)=e^{i\alpha} \frac{z-a}{z-\bar{a}}
$$
and hence
$$
f^{-1}(z)=\frac{a-\bar{a}e^{-i\alpha}z}{1-e^{-i\alpha}z},
$$
where $\alpha$ is a real constant.
Since $j$-metric is invariant under translations, stretchings of $\UH$ onto itself and rotations of $\BB$ onto itself, we may assume that $a=i$ and $\alpha=0$.
Then we have
$$
f(z)=\frac{z-i}{z+i}\,\,\,\,\,\,\,\,\mbox{and} \,\,\,\,\,\,\,\
f^{-1}(z)=i\frac{1+z}{1-z}.
$$
By Theorem \ref{sth1} and Theorem \ref{sth2}, we get the results immediately.
\end{proof}
%==============================================================================

\medskip

%==============================================================================
\begin{lemma}\label{1-3}
Let $m\in(1,\infty)$, $\theta\in(0,\frac{\pi}{2m})$, and $r\in(0,1)$. Then\\
(1) $f_1(\theta)\equiv\frac{m\sin\theta}{1-\sin\theta}-\frac{\sin m\theta}{1-\sin m\theta}$ is decreasing from $(0,\frac{\pi}{2m})$ onto $(-\infty,0)$;\\
(2) $f_2(r)\equiv\frac{1-(1-\sin\theta)r}{1-(1-\sin m\theta)r^m}$ is decreasing from  $(0,1)$ onto $(\frac{\sin\theta}{\sin m\theta},1)$;\\
(3) $f_3(r)\equiv\frac{\log(1+\frac{1-r^m}{r^m\sin m\theta})}{\log(1+\frac{1-r}{r\sin \theta})}$ is decreasing from $(0,1)$ onto $(\frac{m\sin\theta}{\sin m\theta},m)$.
\end{lemma}
%==============================================================================
%==============================================================================
\begin{proof}
(1) By differentiation,
$$f'_1(\theta)=m\left(g(1)-g(m)\right),$$
where $g(m)=\frac{\cos m\theta}{(1-\sin m\theta)^2}$.
Since
$$g'(m)=\frac{\theta(1-\sin m\theta+\cos^2m\theta)}{(1-\sin m\theta)^3}>0,$$
we have $f'_1(\theta)<0$ and hence $f_1$ is decreasing on $(0,{\pi}/{2m})$. The limiting values are clear.

\medskip

(2) By differentiation,
\begin{eqnarray*}
& &\left(1-(1-\sin m\theta)r^m\right)^2f'_2(r)\\
&=&-(m-1)(1-\sin\theta)(1-\sin m\theta)r^m+m(1-\sin m\theta)r^{m-1}-(1-\sin\theta)\\
&\equiv& h(r).
\end{eqnarray*}
Since
$$h'(r)=m(m-1)(1-\sin m\theta)r^{m-2}\left(1-r(1-\sin\theta)\right)>0,$$
we have $h$ is increasing on $(0,1)$ and hence $h(r)\leq h(1)$.
By (1),
$$h(1)=(1-\sin\theta)(1-\sin m\theta)f_1(\theta)<0.$$
Therefore, $f'_2(r)<0$ and hence $f_2$ is decreasing on $(0,1)$. The limiting values are clear.

\medskip

(3) Let $f_3(r)\equiv\frac{g_1(r)}{g_2(r)}$, here $g_1(r)=\log(1+\frac{1-r^m}{r^m\sin m\theta})$ and $g_2(r)=\log(1+\frac{1-r}{r\sin \theta})$. Then
$g_1(1^-)=g_2(1^-)=0$, by differentiation, we have
$$\frac{g'_1(r)}{g'_2(r)}=m f_2(r).$$
Therefore, $f_3$ is decreasing on $(0,1)$ by Lemma \ref{lhr}. By l'H${\rm \hat{o}}$pital's Rule and (2) we easily get the limiting values.
\end{proof}
%==============================================================================

\medskip

The planar angular domain is defined as
$$S_{\varphi}=\{r e^{i\theta}\in\mathbb{C}:0<\theta<\varphi\,,\,r>0\}.$$

%==============================================================================
\begin{lemma}\label{rayl}
Let $f:S_{\pi/m}\rightarrow \UH$ with $f(z)=z^m \, (m\geq1)$. Let $x,y\in S_{\pi/m}$ with $\arg(x)=\arg(y)=\theta$. Then for all $x\,,y\in S_{\pi/m}$
\begin{eqnarray*}
\frac{m\sin \theta}{\sin m\theta}j_{S_{\pi/m}}(x,y)\leq j_{\UH}(f(x),f(y))\leq m j_{S_{\pi/m}}(x,y).
\end{eqnarray*}
\end{lemma}
%==============================================================================
%==============================================================================
\begin{proof}
Since $j-$metric is invariant under stretchings, we may assume that $x=r e^{i\theta}$ and $y=e^{i\theta}$, $0<r<1$. By symmetry, we also assume $0<\theta\leq\frac{\pi}{2m}$. Then
$$\frac{j_{\UH}(f(x),f(y))}{j_{S_{\pi/m}}(x,y)}=\frac{\log(1+\frac{1-r^m}{r^m\sin m\theta})}{\log(1+\frac{1-r}{r\sin \theta})}.$$
By Lemma \ref{1-3}(3), we get the result.
\end{proof}
%==============================================================================
\medskip

%==============================================================================
\begin{lemma}\label{n}
Let $n\in\mathbb{N}$, $0<\theta\leq\frac{\pi}{2n}$. Then for $x,y\in \mathbb{C}\setminus\{0\}$,
\begin{eqnarray*}
1+\frac{|x^n-y^n|}{|x|^n\sin n\theta}\leq \left(1+\frac{|x-y|}{|x|\sin\theta}\right)^n.
\end{eqnarray*}
\end{lemma}
%==============================================================================
%==============================================================================
\begin{proof}
The equality holds if $n=1$.

Next, suppose that the inequality holds when $n=m$ . Namely,
\begin{eqnarray}\label{kI}
1+\frac{|x^m-y^m|}{|x|^m\sin m\theta}\leq \left(1+\frac{|x-y|}{|x|\sin\theta}\right)^m,
\end{eqnarray}
where $0<\theta\leq\frac{\pi}{2m}$.

Then, if $n=m+1$, we have  $0<\theta\leq\frac{\pi}{2(m+1)}<\frac{\pi}{2m}$ and by (\ref{kI}),
\begin{eqnarray*}
\left(1+\frac{|x-y|}{|x|\sin\theta}\right)^{m+1}&\geq&\left(1+\frac{|x^m-y^m|}{|x|^m\sin m\theta}\right)\left(1+\frac{|x-y|}{|x|\sin\theta}\right)\\
&\geq& 1+\frac{|x-y|}{|x|\sin\theta}+\frac{|x^m-y^m|}{|x|^m\sin m\theta}\left(1+\frac{|x-y|}{|x|}\right)\\
&\geq& 1+\frac{|x^{m+1}-x^{m}y|}{|x|^{m+1}\sin\theta}+\frac{|x^my-y^{m+1}|}{|x|^{m+1}\sin m\theta}\\
&\geq& 1+\frac{|x^{m+1}-y^{m+1}|}{|x|^{m+1}\sin (m+1)\theta}.
\end{eqnarray*}
This completes the proof by induction.
\end{proof}
%==============================================================================
\medskip

%=====================================================================
\begin{theorem}\label{sth4}
Let $f:S_{\pi/m}\rightarrow \UH$ with $f(z)=z^m\,(m\in\mathbb N)$. Then for all $x,y\in S_{\pi/m}$,
$$j_{\UH}(f(x),f(y))\leq m j_{S_{\pi/m}}(x,y),$$
and the constant $m$ is the best possible.
\end{theorem}
%==============================================================================
%==============================================================================
\begin{proof}
By symmetry, we may assume that $d(f(x),\partial \UH)\leq
d(f(y),\partial \UH)$. Hence by Lemma \ref{n}, we obtain
\begin{eqnarray*}
j_{\UH}(f(x),f(y))&=&\log\left(1+\frac{|x^m-y^m|}{|x|^m\sin m\theta}\right)\\
&\leq& m\log\left(1+\frac{|x-y|}{|x|\sin\theta}\right)\\
&\leq& m j_{S_{\pi/m}}(x,y),
\end{eqnarray*}
where $0<\theta=\min\{\arg(x), \frac{\pi}{m}-\arg(x)\}\le
\frac{\pi}{2m}$.

Let $x=r e^{i\alpha}$ and $y=e^{i\alpha}$, where $0<\alpha<\frac{\pi}{2m}$ and $0<r<1$. Letting $r\rightarrow0$, and
by Lemma \ref{1-3}(3) and Lemma \ref{rayl}, we know that the constant $m$ is the best possible.
\end{proof}
%==============================================================================

%==============================================================================
\begin{lemma} \label{le4.5}
 For $z\in\mathbb{C}$ and $p\in\mathbb N$ we have
$$
\log(1+|z^p-1|)\le p\log(1+|z-1|).
$$
\end{lemma}
%==============================================================================
%==============================================================================
\begin{proof}
Putting $z=u+1$, we obtain
$$
1+|z^p-1|=1+\left|\sum_{l=1}^p {p\choose l}u^l\right| \le
1+ \sum_{l=1}^p {p\choose l}|u|^l = (1+|u|)^p = (1+|z-1|)^p,
$$
and  the claim follows.
\end{proof}
%==============================================================================

%==============================================================================
\begin{lemma} \label{le3.2}
Let $Q_d^*$ denote the class of all polynomials with exact degree
$d(\ d\ge 1)$ which have no zeros inside the disk $\BB$.

If $Q\in Q_d^*$ and $u,v\in \BB$, then
$$
\left|\frac{Q(u)}{Q(v)}-1\right|\le \left(1+\frac{|u-v|}{1-|v|}\right)^d-1.
$$
\end{lemma}
%==============================================================================

%==============================================================================
\begin{proof} \ A polynomial $Q(z)$ with zeros $\{-a_l\}$($l=1,2,\cdots,d$)
has a representation in the form
$$
Q(z)=C\prod_{l=1}^d (z+a_l),\,\, C\neq 0.
$$

Since $Q(z)\neq 0$ for $z\in \BB$, we have $|a_l|\ge
1$.

Therefore,
\begin{eqnarray*}
\left|\frac{Q(u)}{Q(v)}-1\right|&=&\left|\prod_{l=1}^d\frac{u+a_l}{v+a_l}-1\right|\\
&=&\left|\prod_{l=1}^d \left(1+\frac{u-v}{a_l+v}\right)-1\right|\\
&=&\left|\sum \prod(\cdot)(\cdot)\right|\le \sum\prod|(\cdot)||(\cdot)|\\
&=&\prod_{l=1}^d \left(1+\frac{|u-v|}{|a_l+v|}\right)-1\\
&\le&\prod_{l=1}^d \left(1+\frac{|u-v|}{|a_l|-|v|}\right)-1\\
&\le&\left(1+\frac{|u-v|}{1-|v|}\right)^d-1.
\end{eqnarray*}
\end{proof}
%==============================================================================

%==============================================================================
\begin{remark} \label{rmk3.3}
Note that, since $\left|\frac{Q(u)}{Q(v)}-1\right|\ge \left|\frac{Q(u)}{Q(v)}\right|-1$,
we also obtain
$$
\left|\frac{Q(u)}{Q(v)}\right|\le \left(1+\frac{|u-v|}{1-|v|}\right)^d.
$$
\end{remark}
%==============================================================================

%==============================================================================
\begin{nonsec}
\begin{proof}[\bf Proof of Theorem \ref{sth5}]
For $x,\,y\in \BB\setminus \{0\}$ , we have
$$
j_{\BB\setminus \{0\}}(x,y)=\log\left(1+\frac{|x-y|}{\min\{|x|,|y|,1-|x|,1-|y|\}}\right)
$$
and
$$
j_{\BB\setminus \{0\}}(f(x),f(y))=\log\left(1+\frac{|f(x)-f(y)|}{T}\right),
$$
where $T= \min\{|f(x)|,|f(y)|,1-|f(x)|,1-|f(y)|\}$.

{\it Case 1.} $T=1-|f(x)|$.
Since
\begin{eqnarray*}
|f(x)-f(y)|&=&|x-y|\left|\sum_{l=1}^{p} a_l(\sum_{\nu+\mu=l-1} x^\nu y^\mu)\right|\\
&\le&|x-y|\sum_{l=1}^{\infty} |a_l|(\sum_{\nu=0}^{l-1} |x|^\nu)
\end{eqnarray*}
and
\begin{eqnarray*}
1-|f(x)|&\ge& \sum_{l=1}^{p} |a_l|-\sum_{l=1}^{p} |a_l||x|^l\\
&=&(1-|x|)\sum_{l=1}^{p} |a_l|(\sum_{\nu=0}^{l-1} |x|^\nu),
\end{eqnarray*}
we obtain
\begin{eqnarray*}
j_{\BB\setminus \{0\}}(f(x),f(y))&\le& \log\left(1+\frac{|x-y|}{1-|x|}\right)\\
&\le&j_{\BB\setminus \{0\}}(x,y).
\end{eqnarray*}

{\it Case 2.} $T=1-|f(y)|$.
This case can be treated in the same way as Case 1.

{\it Case 3.} $T=|f(y)|$.
Now we assume that $0$ is $m$-th order zero of $f$. Since $f$
has no other zeros in $\BB$ by the assumption of this theorem, we get
$f(z)=z^m Q(z), \ \ Q\in
Q_d^*, \ m+d=p$,  here $Q_d^*$ is as in Lemma \ref{le3.2}.

By Lemma \ref{le3.2} and Remark \ref{rmk3.3}, it follows
\begin{eqnarray*}
\frac{|f(x)-f(y)|}{|f(y)|}&=&\left|\frac{x^m Q(x)}{y^mQ(y)}-1\right|\\
&=&\left|\left(\frac{x^m}{y^m}-1\right)\frac{Q(x)}{Q(y)}+\frac{Q(x)}{Q(y)}-1\right|\\
&\le&\left|\frac{x^m}{y^m}-1\right|\left|\frac{Q(x)}{Q(y)}\right|+\left|\frac{Q(x)}{Q(y)}-1\right|\\
&\le&\left(1+\left|\frac{x^m}{y^m}-1\right|\right)\left(1+\frac{|x-y|}{1-|y|}\right)^d-1.
\end{eqnarray*}

Therefore, by Lemma \ref{le4.5}, we have
\begin{eqnarray*}
j_{\BB\setminus \{0\}}(f(x),f(y))&=&\log\left(1+\frac{|f(x)-f(y)|}{|f(y)|}\right)\\
&\le& d\log\left(1+\frac{|x-y|}{1-|y|}\right)+\log\left(1+\left|\frac{x^m}{y^m}-1\right|\right)\\
&\le& d\log \left(1+\frac{|x-y|}{1-|y|}\right)+m\log\left(1+\frac{|x-y|}{|y|}\right)\\
&\le& p j_{\BB\setminus \{0\}}(x,y),
\end{eqnarray*}
and the proof for case 3 is done.

{\it Case 4.} $T=|f(x)|$.
This case can be treated in the same way as Case 3.

For the sharpness of the inequality, let $f(z)=z^p\,( \ p\in\mathbb N)$. For $s,\,t\in(0,\frac 12)$ and $s<t$, we have
\begin{eqnarray*}
j_{\BB\setminus \{0\}}(f(t),f(s))=\log\left(1+\frac{t^p-s^p}{s^p}\right) =p
\log\left(\frac{t}{s}\right)= pj_{\BB\setminus\{0\}}(t,s).
\end{eqnarray*}
Therefore the constant $p$ is sharp.
\end{proof}
\end{nonsec}
%==============================================================================
\medskip

%==============================================================================
\begin{nonsec}
\begin{proof}[\bf Proof of Theorem \ref{mth3}]
Let $r=\max\{|x|,|y|\}$ and suppose that $|f(x)|\ge
|f(y)|$. Then
$$
j_{\BB}(x,y)=\log\left(1+\frac{|x-y|}{1-r}\right)
$$
and
$$
j_{\BB}(f(x),f(y))=\log\left(1+\frac{|f(x)-f(y)|}{1-|f(x)|}\right).
$$
The next two inequalities follow by the proof of Theorem \ref{sth5} Case 1 (replace $p$ with $\infty$):
\begin{eqnarray*}
|f(x)-f(y)|%&=&|x-y||\sum_{l=1}^{\infty} a_l(\sum_{\nu+j=l-1} x^\nuy^j)|\\
\le|x-y|\sum_{l=1}^{\infty} |a_l|(\sum_{\nu=0}^{l-1} |x|^\nu)
\end{eqnarray*}
and
\begin{eqnarray*}
1-|f(x)|%&\ge& \sum_{l=1}^{\infty} |a_l|-\sum_{l=1}^{\infty} |a_l||x|^l\\
\ge(1-|x|)\sum_{l=1}^{\infty} |a_l|(\sum_{\nu=0}^{l-1} |x|^\nu).
\end{eqnarray*}

Therefore,
\begin{eqnarray*}
j_{\BB}(f(x),f(y))\le
\log\left(1+\frac{|x-y|}{1-|x|}\right)\le j_{\BB}(x,y).
\end{eqnarray*}

Obviously, the equality holds for the identity map and hence the sharpness is clear.
\end{proof}
\end{nonsec}
%==============================================================================
\medskip

%==============================================================================
\begin{remark}
{\bf 1.} {\rm Let $f:\BB\rightarrow \BB$ be an analytic function. According to the Schwarz lemma, for all $x\,,y\in\BB$ we have
$$\rho_{\BB}(f(x),f(y))\leq \rho_{\BB}(x,y)\,,$$
and further, by the inequality \eqref{jrho1},
\beq\label{jb}
j_{\BB}(f(x),f(y))\leq 2 j_{\BB}(x,y)\, .
\eeq
Note that Theorem \ref{mth3} gives a sufficient
condition for the mapping to be a contraction, i.e., to have
the Lipschitz constant at most $1\,.$ }

{\bf 2.} {\rm Following \cite[Example 3.10]{v0} consider the
exponential function $f: \BB\rightarrow \BB\setminus\{0\}=f\BB$ with $f(z)=\exp\left(\frac{z+1}{z-1}\right)\,.$ Let $z_l=\frac{e^l-1}{e^l+1}\,(l\in \mathbb{N})$. Then
$$j_{\BB}(z_l,z_{l+1})=\log\frac{e^{l+1}+1}{e^l+1}\rightarrow 1\,\,\,(l\rightarrow \infty)$$
and
$$j_{f\BB}(f(z_l),f(z_{l+1}))=e^l(e-1)\rightarrow\infty\,\,\,(l\rightarrow \infty).$$
Therefore, in (\ref{jb}) $j_{\BB}(f(x),f(y))$ can not be replaced with $j_{f\BB}(f(x),f(y))$ if $f\BB\subset\BB$ has isolated boundary points.

Furthermore, we have $|f(0)|+|f'(0)|=3/e>1$ and for $0<t<1$
\begin{eqnarray*}
 \lim_{t\rightarrow 0}\frac{j_{\BB}(f(0),f(t))}{j_{\BB}(0,t)}&=&\lim_{t\rightarrow 0}\frac{\log\left(1+\frac{e^{-1}-e^{(t+1)/(t-1)}}{1-e^{-1}}\right)}{\log\left(1+\frac{t}{1-t}\right)}\\
 %&=&\frac{e}{e-1}\lim_{t\rightarrow 0}\frac{e^{-1}-e^{(t+1)/(t-1)}}{t}\\
 &=&\lim_{t\rightarrow 0}\frac{2}{1-t}\frac{1}{e^{-(t+1)/(t-1)}-1}\\
 &=&\frac{2}{e-1}>1.
\end{eqnarray*}
Therefore, in Theorem \ref{mth3} the hypothesis $\sum_{l=0}^{\infty}|a_l|\le 1$ cannot be removed.
}
\end{remark}
%==============================================================================
\medskip
%==============================================================================
\begin{theorem}\label{sth6}
Let $f(z)=a+r^2\frac{(z-a)}{|z-a|^2}$ be the inversion in $S^{n-1}(a,r)$ with ${\rm Im}\,a=0$. Then $f(\Hn)=\Hn$ and for all $x,y\in \Hn$,
$$j_{\Hn}(f(x),f(y))\leq 2 j_{\Hn}(x,y). $$
The constant 2 is the best possible.
\end{theorem}
%==============================================================================
%==============================================================================
\begin{proof}
The inequality is clear by Theorem \ref{gplem}. Without loss of generality, we may assume that $a=0$ and $r=1$.

Putting $x=i$, $y=t+i$, $t>0$, we have
$$\lim_{t\to\infty}\frac{j_{\Hn}(f(i),f(t+i))}{j_{\Hn}(i,t+i)}=\lim_{t\to\infty}\frac{\log(1+t\sqrt{1+t^2})}{\log(1+t)}=\lim_{t\to\infty}\frac{\log t^2}{\log t}=2.$$
Hence the constant 2 is the best possible.

\end{proof}
%==============================================================================

%%%%%%%%%%%%%%%
%%%%%%%%%%%%%%%
\section{Proof of Theorem \ref{mth1}}
%%%%%%%%%%%%%%%
%%%%%%%%%%%%%%%

In order to prove Theorem \ref{mth1}, we first need some lemmas.

%==============================================================================
\begin{lemma} \label{le2.01}{\rm \cite[Theorem 3.5.1]{b}}
Let $f$ be a M\"obius transformation and $f(\Bn)=\Bn$. Then
$$f(x)=(\sigma x)A,$$
where $\sigma$ is an inversion in some sphere orthogonal to $S^{n-1}$ and $A$ is an orthogonal matrix.
\end{lemma}
%==============================================================================

%==============================================================================
\begin{lemma}\label{le2.1}
Let $a,b\in \mathbb B^n$. Then

(1)
$$
|a|^2|b-a^*|^2-|b-a|^2=(1-|a|^2)(1-|b|^2);
$$

(2)
$$
\frac{||b|-|a||}{1-|a||b|}\le \frac{|b-a|}{|a||b-a^*|}\le
\frac{|b|+|a|}{1+|a||b|}.
$$
\end{lemma}
%==============================================================================
%==============================================================================
\begin{proof}(1) By calculation, we have
\begin{eqnarray*}
& &|a|^2|b-a^*|^2-|b-a|^2\\
&=&|a|^2\left(|b|^2+\frac{1}{|a|^2}-\frac{2( b\cdot a)}{|a|^2}\right)-\left(|b|^2+|a|^2-2 (b\cdot a)\right)\\
&=&1+|a|^2|b|^2-|a|^2-|b|^2\\
&=&(1-|a|^2)(1-|b|^2).
\end{eqnarray*}

(2) This can be directly obtained by \cite[Exercise 2.52]{v}.
\end{proof}
%==============================================================================
\medskip

%==============================================================================
\begin{lemma} \label{le2.3}
Let $c,d\in (0,1), \theta\in (0,1]$.

(1)$f(\theta)\equiv\frac{\log\left(1+\frac{2cd\theta}{1-cd}\right)}{\log\left(1+\frac{2d\theta}{1-d}\right)}$
is increasing. In particular,
\begin{eqnarray*}
\frac{\log\left(1+\frac{2cd\theta}{1-cd}\right)}{\log\left(1+\frac{2d\theta}{1-d}\right)}
\le\frac{\log\left(1+\frac{2cd}{1-cd}\right)}{\log\left(1+\frac{2d}{1-d}\right)}.
\end{eqnarray*}

(2) $g(\theta)\equiv\frac{{ \arth }(c\theta)}{{\arth}\theta}$ is
decreasing. In particular,
\begin{eqnarray*}
\frac{{ \arth }(c\theta)}{{ \arth }
\theta}\le c.
\end{eqnarray*}

(3)
$$
  \left(1+\frac{2cd\theta}{1-cd}\right)\left(1+\frac{c(1-d)}{1+cd}\right)
\le 1+\frac{c(1-d)+2cd\theta}{1-cd}\,.
$$
\end{lemma}
%==============================================================================

%==============================================================================
\begin{proof} (1)
Let $f_1(\theta)=\log\left(1+\frac{2cd\theta}{1-cd}\right)$ and
$f_2(\theta)=\log\left(1+\frac{2d\theta}{1-d}\right)$. Then we have
$f_1(0^+)=f_2(0^+)=0$ and
$$
\frac{f'_1(\theta)}{f'_2(\theta)}=1-\frac{1-c}{1-cd+2cd\theta},
$$
which is clearly increasing in $\theta$. Therefore,
the monotonicity of $f$ immediately follows by Lemma \ref{lhr}. The inequality follows by the monotonicity of $f$.
\medskip

(2) Let $g_1(\theta)={ \arth }(c\theta)$ and
$g_2(\theta)={\arth}\theta$. Then we have
$g_1(0^+)=g_2(0^+)=0$ and
$$
\frac{g'_1(\theta)}{g'_2(\theta)}=\frac{1}{c}\left(1-\frac{1-c^2}{1-c^2\theta^2}\right),
$$
which is clearly decreasing in $\theta$. Therefore,
the monotonicity of $g$ follows by Lemma \ref{lhr}. The inequality immediately follows by the monotonicity of $g$ and l'H$\rm\hat{o}$pital's Rule .
\medskip

(3) This inequality can be easily proved by direct calculation.
\end{proof}
%==============================================================================
\medskip

Now we are in a position to give a short proof of
Theorem \ref{mth1}.

%==============================================================================
\begin{nonsec}{\bf Proof of Theorem \ref{mth1}.}
{\rm The claim is trivial for $a=0$, therefore, we only need to consider $a\neq 0$.
Since $j$-metric is invariant under orthogonal transformations and by Lemma \ref{le2.01}, for $x,y,a\in \mathbb B^n$, we have
$$j_{\mathbb B^n}(f(x),f(y))=j_{\mathbb B^n}(\sigma_a(x),\sigma_a(y)),$$
where $\sigma_a(x)$ is an inversion in the sphere $S^{n-1}(a^*,\sqrt{|a|^{-2}-1})$ orthogonal to $S^{n-1}$.
Thus, it suffices to estimate the
expression
$$
J(x,y;a)\equiv\frac{j_{\mathbb B^n}(\sigma_a(x),\sigma_a(y))}{j_{\mathbb
B^n}(x,y)}=\frac{\log\left(1+\frac{|\sigma_a(x)-\sigma_a(y)|}{\min\{1-|\sigma_a(x)|,1-|\sigma_a(y)|\}}\right)}
{\log\left(1+\frac{|x-y|}{\min\{1-|x|,1-|y|\}}\right)}.
$$

Let $r=\max\{|x|, |y|\}$ and suppose $|\sigma_a(x)|\ge|\sigma_a(y)|$. Then by (\ref{ivd}), we have
$$
\min\{1-|\sigma_a(x)|,1-|\sigma_a(y)|\}=1-|\sigma_a(x)|=\frac{|a||x-a^*|-|x-a|}{|a||x-a^*|}.
$$

We first prove the right-hand side of the inequality.
By Lemma \ref{le2.1}, we get
\begin{eqnarray*}
j_{\Bn}(\sigma_a(x),\sigma_a(y))&=&\log\left(1+\frac{(1-|a|^2)|x-y|}{|a||y-a^*|(|a||x-a^*|-|x-a|)}\right)\\
&=&\log\left(1+\frac{|x-y|(|a||x-a^*|+|x-a|)}{|a||y-a^*|(1-|x|^2)}\right)\\
&=&\log\left(1+\frac{|x-y||x-a^*|}{(1-|x|^2)|y-a^*|}\left(1+\frac{|x-a|}{|a||x-a^*|}\right)\right)\\
&\le&\log\left(1+\frac{|x-y|}{1-r^2}\left(1+\frac{|x-y|}{|y-a^*|}\right)\left(1+\frac{|x|+|a|}{1+|a||x|}\right)\right)\\
&\le&\log\left(1+\frac{|x-y|}{1-r}\left(1+\frac{|a||x-y|}{1-|a|r}\right)\left(1+\frac{|a|(1-r)}{1+|a|r}\right)\right).
\end{eqnarray*}
Then
\begin{eqnarray*}
J(x,y;a)&\le&
\frac{\log\left(1+\frac{|x-y|}{1-r}\left(1+\frac{|a||x-y|}{1-|a|r}\right)\left(1+\frac{|a|(1-r)}{1+|a|r}\right)\right)}{\log\left(1+\frac{|x-y|}{1-r}\right)}\\
&=&\frac{\log\left(1+\frac{2r\theta}{1-r}\left(1+\frac{2|a|r\theta}{1-|a|r}\right)\left(1+\frac{|a|(1-r)}{1+|a|r}\right)\right)}{\log\left(1+\frac{2r\theta}{1-r}\right)},
\end{eqnarray*}
where $\theta=\frac{|x-y|}{2r}$.

By  Lemma \ref{le2.3},
it follows
\begin{eqnarray*}
J(x,y;a)&\le&
\frac{\log\left(1+\frac{2r\theta}{1-r}\left(1+\frac{|a|(1-r)}{1-|a|r}+\frac{2|a|r\theta}{1-|a|r}\right)\right)}{\log\left(1+\frac{2r\theta}{1-r}\right)}\\
&=&1+\frac{\log\left(1+\frac{2|a|r\theta}{1-|a|r}\right)}{\log\left(1+\frac{2r\theta}{1-r}\right)}\\
&\le&1+\frac{\log\left(1+\frac{2|a|r}{1-|a|r}\right)}{\log\left(1+\frac{2r}{1-r}\right)}\\
&=&1+\frac{{ \arth }(|a|r)}{{ \arth }
r}\\
&\le& 1+|a|.
\end{eqnarray*}
Therefore, we get
\begin{eqnarray*}
j_{{\mathbb B}^n}(f(x),f(y)) \leq (1+ |a|) j_{{\mathbb B}^n}(x,y).
\end{eqnarray*}

The sharpness of the upper bound $1+|a|$ was proved in \cite[Remark 3.4]{kvz} by taking $x=t a/|a|=-y,\,t\in(0, 1)$, and letting $t\rightarrow 0^+$.

For the left-hand side of the inequality, we
note that $f^{-1}(x)=A^{-1}\sigma^{-1}_a(x)=A^{-1}\sigma_{a}(x)$, here $\sigma_{a}(x)$ and $A$ are as above. Note that because $A$ is an orthogonal matrix, so is $A^{-1}$. Then by the above proof, for $x,\,y\in\Bn$,
we get
$$
\frac{j_{\mathbb B^n}(f^{-1}(x),f^{-1}(y))}{j_{\mathbb B^n}(x,y)}
=\frac{j_{\mathbb B^n}(\sigma_{a}(x),\sigma_{a}(y))}{j_{\mathbb B^n}(x,y)}\le 1+|a|.
$$
Therefore, we have
\begin{eqnarray}\label{mth1i}
j_{{\mathbb B}^n}(f(x),f(y))\geq \frac{1}{1+|a|}j_{{\mathbb B}^n}(x,y).
\end{eqnarray}
This completes the proof.
}
\hfill$\square$
\end{nonsec}
%==============================================================================
\medskip

%==============================================================================
\begin{conjecture}
Let $a\in\Bn$ and $f: \Bn\rightarrow\Bn=f\Bn$ be a M\"obius transformation with $f(0)=a$. Then
for  $x,\,y\in \Bn\setminus \{0\}$
\begin{eqnarray*}
j_{\Bn\setminus \{a\}}(f(x),f(y))\le C(a) j_{\Bn\setminus \{0\}}(x,y),
\end{eqnarray*}
where the constant $C(a)=1+(\log\frac{2+|a|}{2-|a|})/\log3$ is the best possible.
\end{conjecture}
%==============================================================================
\medskip

\subsection*{Acknowledgments}
The research of Matti Vuorinen was supported by the Academy of Finland,
Project 2600066611. The research of Gendi Wang was supported by CIMO
of Finland, Grant TM-10-7364.

\end{document}